\documentclass{article}
\usepackage{epsfig}
\usepackage{amsmath}
\usepackage{amssymb}
\usepackage{amsfonts}
\usepackage{float}
\usepackage{graphics}

\newcommand{\cp}{{\,\Box\,}}

\newtheorem{thm}{Theorem}
\newtheorem{prop}{Proposition}

\newtheorem{conj}{Conjecture}

\newtheorem{prob}{Problem}

\def\node(#1){\put(#1){\circle*{2}}}
\def\nodeo(#1){\put(#1){\circle{2}}}
\def\vert(#1){\put(#1){\circle*{1.5}}}
\def\verto(#1){\put(#1){\circle{1.5}}}
\def\lab(#1)#2{\put(#1){\makebox(0,0)[c]{#2}}}
\newcommand{\cF}{\mathcal{F}}

\newcommand{\qed}{\hfill$\Box$}

\usepackage{graphicx}
\usepackage{tikz}
\title{On maximizing private neighbors in graphs}

\author{
Stephen T. Hedetniemi $^{a,}$\thanks{Email: \texttt{hedet@clemson.edu}}
\and
Douglas F. Rall $^{b,}$\thanks{Email: \texttt{doug.rall@furman.edu}}
}

\begin{document}

\date{\today}

\maketitle

\begin{center}
	$^a$  Emeritus Professor of Computer Science, Clemson University, Clemson, SC, USA\\
	\medskip

	$^b$ Emeritus Professor of Mathematics, Furman University, Greenville, SC, USA \\
   \medskip

\end{center}

\begin{abstract}
Given a set $U \subset V$ of vertices in a graph $G = (V, E)$, a {\it private neighbor with respect to the set $U$} is any vertex $w \in V$ having precisely one neighbor, say $v$, in $U$.  If $w \in V - U$, then $w$ is called an {\it external private neighbor} of $v$ with respect to $U$.  If $w \in U$ then $w$ is called an {\it internal private neighbor} of $v$ with respect to $U$.  We also add one special case:  if $w \in U$ and $N(w) \cap U = \emptyset$, then we say that $w$ is a {\it self private neighbor} with respect to $U$.  By definition, a self private neighbor with respect to $U$ is an isolated vertex in the subgraph of $G$ induced by $U$.  In this paper we consider the general problems of trying to find sets of vertices which maximize the number of private neighbors of specific types in a graph.  In the process of doing this we define several new maximization parameters of graphs which generalize some known and well-studied parameters of graphs relating to vertex and edge independence, domination and irredundance in graphs.

\end{abstract}


\noindent
{\bf AMS Subj.\ Class.\ (2020)}: 05C69.

\section{Introduction}

Let $G=(V,E)$ be a finite, simple graph of order $n(G)$ and size $m(G)$.  In general we follow the notation and definitions of~\cite{hhh-2023}.  We include here those used most often in this study.  For $x\in V(G)$  the \emph{open neighborhood} of $x$ is the set $N_G(x)$ defined by $N_G(x)=\{w \in V\,:\,xw \in E\}$.  Any  $w \in N(x)$ is 
called a \emph{neighbor} of $x$ and we say that vertices $w$ and $x$ are \emph{adjacent}.  The \emph{closed neighborhood} of $x$ is the set $N_G[x]$ defined by $N_G[x]=N_G(x) \cup \{x\}$.  Let $U \subseteq V$.  The \emph{open neighborhood} of $U$ is the set $N_G(U)=\bigcup_{x\in U} N_G(x)$ and its \emph{closed neighborhood} is the set $N_G[U]=N_G(U) \cup U$. We omit the subscripts from these notations when the graph $G$ is clear from the context.  The set $U$  is a \emph{dominating set} of $G$ if $|N(v) \cap U|\ge 1$ for every vertex $v \in V - U$;  equivalently, if $N[v] \cap U \neq \emptyset$ for every vertex $v \in V$. The
\emph{domination number} of $G$, denoted by $\gamma(G)$, is the minimum cardinality of a
dominating set in $G$, while the \emph{upper domination number}, $\Gamma(G)$, is the maximum cardinality of a minimal dominating set in $G$. A dominating set $U$ is a \emph{perfect} dominating set if $|N[x] \cap U|=1$ for every $x \in V-U$ and is called an \emph{efficient} dominating set if $|N[x] \cap U|=1$ for every $x \in V$.  If $n$ is a positive integer, the set of positive integers not larger than 
$n$ will be denoted $[n]$. 

The Cartesian product, $G \cp H$, of two graphs $G=(V_1,E_1)$ and $H=(V_2,E_2)$ has vertex set $V_1 \times V_2$.  Two vertices in
$G \cp H$ are adjacent if they are equal in one coordinate and adjacent in the other.  To simplify notation, we let $G_{n,m}=P_n\cp P_m$ denote the grid graph having $m$ rows and $n$ columns.   Let $p$ and $q$ be positive integers.  A \emph{double star}, denoted $S(p,q)$, is the tree obtained from a path $u,v$ of order $2$ by adding $p$ vertices adjacent to $u$ and $q$ vertices adjacent to $v$.  The vertices 
$u$ and $v$ are called the \emph{centers} of $S(p,q)$.

\section{Private neighbors} \label{sec:privneighbors}

In 1978, Cockayne, Hedetniemi and Miller~\cite{Cockayne78} introduced the concept of private neighbors in graphs, as follows.
Given a set $U \subset V$ of vertices in a graph $G$, a \emph{private neighbor with respect to the set $U$}, is any vertex $w \in V$ having precisely one neighbor in $U$. That is, $w$ is a private neighbor with respect to $U$ if and only if $|N(w) \cap U|=1$. If $w \in V - U$ and $N(w) \cap U=\{v\}$, then $w$ is called an \emph{external private neighbor} of $v$ with respect to $U$. If $w \in U$ and $N(w) \cap U=\{v\}$, then $w$ is called an \emph{internal private neighbor} of $v$ with respect to $U$.  In other words, $w$ is an internal private neighbor of $v\in U$ with respect to $U$ if $w$ is adjacent to $v$ and $w$ has degree $1$ in the subgraph of $G$ induced by $U$.  In addition, we add one special case.  If $w \in U$ and $N[w] \cap U = \{w\}$, then we say that $w$ is a \emph{self private neighbor} with respect to $U$.  By definition, a self private neighbor with respect to $U$ is an isolated vertex in the subgraph $G[U]$ induced by $U$.  For each type of private neighbor defined here we will omit the phrase ``with respect to $U$'' if the set $U$ is clear from the context.

Thus, with respect to any set $U \subset V$ in a graph $G$, we can associate three numbers $PN(U) = (S(U), I(U), E(U))$, where (i) $S(U)$ is the number of self private neighbors with respect to $U$, (ii) $I(U)$ is the number of internal private neighbors with respect to $U$, and (iii) $E(U)$ is the number of external private neighbors with respect to $U$.  These three types of private neighbors, therefore, give rise to seven types of sets $U$ depending on the types of private neighbors that every vertex in $U$ must have.  For sake of reference
we will refer to these  as being ``irredundance-type'' invariants.
\begin{enumerate}

\item Every vertex in $U$ is a self private neighbor. Such sets are called \emph{independent}, and the \emph{vertex independence number},
    $\alpha(G)$, is the maximum cardinality of an independent set in $G$.
\item  Every vertex in $U$ has an internal private neighbor.   Such sets correspond to \emph{strong matchings} in graphs, the induced subgraphs of which consist of disjoint unions of complete subgraphs of order $2$.
    The \emph{strong matching number}, $\alpha^*(G)$, equals the maximum cardinality of a strong matching in $G$.
\item  Every vertex in $U$ has an external private neighbor. Equivalently, $|N(v)-N[U-\{v\}]|\ge 1$ for every $v \in U$.
Such sets are called \emph{open irredundant sets}.
The \emph{upper open irredundance number}, $OIR(G)$, is the maximum cardinality of an open irredundant set in $G$.
\item Every vertex in $U$ has an external private neighbor or is a self private neighbor.  Equivalently, $|N[v]-N[U-\{v\}]|\ge 1$ for every $v \in U$.  
Such sets are called \emph{irredundant}, and the \emph{upper irredundance number}, $IR(G)$, is the maximum cardinality of an irredundant set in G.
\item  Every vertex in $U$ has  an external private neighbor or an internal private neighbor. 
Equivalently, $|N(v)-N(U-\{v\})|\ge 1$ for every $v \in U$.
Such sets are called \emph{open-open irredundant sets}.  The \emph{upper open-open irredundance number}, $OOIR(G)$, equals the maximum cardinality of an open-open irredundant set in $G$.
\item  Every vertex in $U$  has an internal private neighbor or is a self private neighbor. Such sets are called 
\emph{$1$-dependent sets}, meaning that the maximum degree of the vertices in the subgraph induced by $U$ is at most $1$.
The {\it $1$-dependence number}, $\alpha_1(G)$, equals the maximum cardinality of a $1$-dependent set in $G$.  
\item  Every vertex in $U$ has at least one private neighbor of some kind, either external, internal or self.  
Equivalently, $|N[v]-N(U-\{v\})|\ge 1$ for every $v \in U$.
Such sets are called   \emph{closed-open irredundant sets}.  The \emph{upper closed-open irredundance number}, $COIR(G)$, equals the maximum cardinality of a closed-open irredundant set in $G$.
\end{enumerate}

The concept of irredundance is now well studied, there being more than 200 papers on this, a representative sample being the following \cite{Bollobas79, Cockayne99, Farley83, Finbow03, Hedetniemi21, Mynhardt21}. Cameron~\cite{Cameron-1989} first studied strong (or induced) matchings; Fink and Jacobson~\cite{fj-1985} introduced the concept of $1$-dependence; Farley and Shachum~\cite{Farley83} defined open irredundance, open-open irredundance and closed-open irredundance.  See the papers by Fellows, Fricke, Hedetniemi and Jacobs~\cite{Fellows94} and Cockayne~\cite{Cockayne99}
that present all of these types of irredundance as they relate to each other.  

\section{Maximizing private neighbors} \label{sec:mpn}

In this section we define graphical invariants arising from the three types of private neighbors defined in Section~\ref{sec:privneighbors}. Instead of focusing
on the subsets of vertices each of which has some kind of private neighbor, in this new context we are interested in the set of private 
neighbors themselves.
For each of the seven types of private neighbor combinations there correspondents a maximization parameter.

\begin{enumerate}
\item  $SPN(G)$, the maximum number of self private neighbors with respect to a set $U \subseteq V$. Such a set
$U$ is called an $SPN(G)$-set. This is equivalent to the maximum number of vertices of degree zero in a subgraph $G[U]$ induced by a vertex set $U \subseteq V$.
\item $IPN(G)$, the maximum number of internal private neighbors with respect to a set $U \subseteq V$. Such a set
$U$ is called an $IPN(G)$-set. This is equivalent to the maximum number of vertices of degree one in the subgraph $G[U]$ induced by a vertex set $U \subseteq V$.  
\item $EPN(G)$, the maximum number of external private neighbors with respect to a set $U\subseteq V$. Such a set
$U$ is called an $EPN(G)$-set.
\item $ESPN(G)$, the maximum number of external and self private neighbors with respect to a set $U\subseteq V$.  Such a set
$U$ is called an $ESPN(G)$-set. 
\item  $EISPN(G)$, the maximum number of external, internal, and self private neighbors with respect to a set $U\subseteq V$.
Such a set $U$ is called an $EISPN(G)$-set.
\item $EIPN(G)$, the maximum number of external and internal private neighbors with respect to a set $U\subseteq V$.
Such a set $U$ is called an $EIPN(G)$-set.
\item $ISPN(G)$, the maximum number of internal and self private neighbors with respect to a set $U\subseteq V$. Such a set
$U$ is called an $ISPN(G)$-set.
\end{enumerate}

It is easy to see from the definitions above that $SPN(G) = \alpha(G)$ and $ISPN(G)\ge \alpha_1(G)$.  To our knowledge, the other five invariants have not been studied. These maximum private neighbors invariants, with the exception of $SPN$, are illustrated on $P_8 \cp P_3$ in Figure~\ref{fig:PN}.  The black vertices in each case form a set that realizes the labeled parameter value.

\newpage 

\begin{figure}[htb]
\tikzstyle{every node}=[circle, draw, fill=black!0, inner sep=0pt,minimum width=.2cm]
\begin{center}
\begin{tikzpicture}[thick,scale=1.3]
  \draw(0,0) { 
    +(0.00,4.50) -- +(3.89,4.50)
    +(0.00,4.00) -- +(3.89,4.00)
    +(0.00,3.50) -- +(3.89,3.50)
    +(3.89,4.50) -- +(3.89,3.50)
    +(0.00,4.50) -- +(0.00,3.50)
    +(5.00,4.50) -- +(8.89,4.50)
    +(5.00,4.00) -- +(8.89,4.00)
    +(5.00,3.50) -- +(8.89,3.50)
    +(5.00,4.50) -- +(5.00,3.50)
    +(8.89,4.50) -- +(8.89,3.50)
    +(0.56,4.50) -- +(0.56,3.50)
    +(1.11,4.50) -- +(1.11,3.50)
    +(1.67,4.50) -- +(1.67,3.50)
    +(2.22,4.50) -- +(2.22,3.50)
    +(2.78,4.50) -- +(2.78,3.50)
    +(3.33,4.50) -- +(3.33,3.50)
    +(5.56,4.50) -- +(5.56,3.50)
    +(6.11,4.50) -- +(6.11,3.50)
    +(6.67,4.50) -- +(6.67,3.50)
    +(7.22,4.50) -- +(7.22,3.50)
    +(7.78,4.50) -- +(7.78,3.50)
    +(8.33,4.50) -- +(8.33,3.50)
    +(0.00,2.75) -- +(3.89,2.75)
    +(5.00,2.75) -- +(8.89,2.75)
    +(0.00,2.25) -- +(3.89,2.25)
    +(0.00,1.75) -- +(3.89,1.75)
    +(3.89,2.75) -- +(3.89,1.75)
    +(3.33,2.75) -- +(3.33,1.75)
    +(2.78,2.75) -- +(2.78,1.75)
    +(2.22,2.75) -- +(2.22,1.75)
    +(1.67,2.75) -- +(1.67,1.75)
    +(1.11,2.75) -- +(1.11,1.75)
    +(0.56,2.75) -- +(0.56,1.75)
    +(0.00,2.75) -- +(0.00,1.75)
    +(0.00,1.00) -- +(3.89,1.00)
    +(0.00,0.50) -- +(3.89,0.50)
    +(0.00,0.00) -- +(3.89,0.00)
    +(3.89,1.00) -- +(3.89,0.00)
    +(3.33,1.00) -- +(3.33,0.00)
    +(2.78,1.00) -- +(2.78,0.00)
    +(2.22,1.00) -- +(2.22,0.00)
    +(1.67,1.00) -- +(1.67,0.00)
    +(1.11,1.00) -- +(1.11,0.00)
    +(0.56,1.00) -- +(0.56,0.00)
    +(0.00,1.00) -- +(0.00,0.00)
    +(5.00,2.25) -- +(8.89,2.25)
    +(5.00,1.75) -- +(8.89,1.75)
    +(8.89,2.75) -- +(8.89,1.75)
    +(8.33,2.75) -- +(8.33,1.75)
    +(7.78,2.75) -- +(7.78,1.75)
    +(7.22,2.75) -- +(7.22,1.75)
    +(6.67,2.75) -- +(6.67,1.75)
    +(6.11,2.75) -- +(6.11,1.75)
    +(5.56,2.75) -- +(5.56,1.75)
    +(5.00,2.75) -- +(5.00,1.75)
    +(5.00,1.00) -- +(8.89,1.00)
    +(5.00,0.50) -- +(8.89,0.50)
    +(5.00,0.00) -- +(8.89,0.00)
    +(8.89,1.00) -- +(8.89,0.00)
    +(8.33,1.00) -- +(8.33,0.00)
    +(7.78,1.00) -- +(7.78,0.00)
    +(7.22,1.00) -- +(7.22,0.00)
    +(6.67,1.00) -- +(6.67,0.00)
    +(6.11,1.00) -- +(6.11,0.00)
    +(5.56,1.00) -- +(5.56,0.00)
    +(5.00,1.00) -- +(5.00,0.00)
    +(0.00,4.50) node{}
    +(0.56,4.50) node[fill=black]{}
    +(1.11,4.50) node{}
    +(1.67,4.50) node{}
    +(2.22,4.50) node{}
    +(2.78,4.50) node[fill=black]{}
    +(3.33,4.50) node{}
    +(3.89,4.50) node{}
    +(8.89,4.50) node[fill=black]{}
    +(8.33,4.50) node[fill=black]{}
    +(7.78,4.50) node{}
    +(7.22,4.50) node{}
    +(6.67,4.50) node[fill=black]{}
    +(6.11,4.50) node[fill=black]{}
    +(5.56,4.50) node{}
    +(5.00,4.50) node[fill=black]{}
    +(0.00,4.00) node{}
    +(0.00,3.50) node[fill=black]{}
    +(0.56,3.50) node{}
    +(1.11,3.50) node{}
    +(1.67,3.50) node[fill=black]{}
    +(2.22,3.50) node{}
    +(2.78,3.50) node{}
    +(3.33,3.50) node[fill=black]{}
    +(3.89,3.50) node{}
    +(5.00,3.50) node[fill=black]{}
    +(5.56,3.50) node{}
    +(6.11,3.50) node[fill=black]{}
    +(6.67,3.50) node[fill=black]{}
    +(7.22,3.50) node{}
    +(7.78,3.50) node{}
    +(8.33,3.50) node[fill=black]{}
    +(8.89,3.50) node[fill=black]{}
    +(8.89,4.00) node{}
    +(8.33,4.00) node{}
    +(7.78,4.00) node[fill=black]{}
    +(7.22,4.00) node[fill=black]{}
    +(6.67,4.00) node{}
    +(6.11,4.00) node{}
    +(5.56,4.00) node[fill=black]{}
    +(5.00,4.00) node[fill=black]{}
    +(3.89,4.00) node{}
    +(3.33,4.00) node{}
    +(2.78,4.00) node{}
    +(2.22,4.00) node{}
    +(1.67,4.00) node[fill=black]{}
    +(1.11,4.00) node{}
    +(0.56,4.00) node{}
    +(0.00,2.75) node{}
    +(0.56,2.75) node{}
    +(1.11,2.75) node[fill=black]{}
    +(1.67,2.75) node[fill=black]{}
    +(2.22,2.75) node{}
    +(2.78,2.75) node{}
    +(3.33,2.75) node[fill=black]{}
    +(3.89,2.75) node{}
    +(5.00,2.75) node[fill=black]{}
    +(5.56,2.75) node[fill=black]{}
    +(6.11,2.75) node{}
    +(6.67,2.75) node[fill=black]{}
    +(7.22,2.75) node[fill=black]{}
    +(7.78,2.75) node{}
    +(8.33,2.75) node[fill=black]{}
    +(8.89,2.75) node[fill=black]{}
    +(0.00,2.25) node[fill=black]{}
    +(0.00,1.75) node[fill=black]{}
    +(0.00,1.00) node{}
    +(0.00,0.50) node[fill=black]{}
    +(0.00,0.00) node[fill=black]{}
    +(0.56,1.00) node{}
    +(1.11,1.00) node[fill=black]{}
    +(1.67,1.00) node{}
    +(2.22,1.00) node{}
    +(2.78,1.00) node[fill=black]{}
    +(3.33,1.00) node[fill=black]{}
    +(3.89,1.00) node{}
    +(3.89,0.50) node{}
    +(3.33,0.50) node{}
    +(2.78,0.50) node{}
    +(2.22,0.50) node{}
    +(1.67,0.50) node{}
    +(1.11,0.50) node{}
    +(0.56,0.50) node{}
    +(0.56,0.00) node{}
    +(1.11,0.00) node{}
    +(1.67,0.00) node[fill=black]{}
    +(2.22,0.00) node[fill=black]{}
    +(2.78,0.00) node{}
    +(3.33,0.00) node{}
    +(3.89,0.00) node[fill=black]{}
    +(5.00,1.00) node{}
    +(5.00,0.50) node{}
    +(5.00,0.00) node[fill=black]{}
    +(5.56,1.00) node[fill=black]{}
    +(6.11,1.00) node{}
    +(6.67,1.00) node{}
    +(7.22,1.00) node[fill=black]{}
    +(7.78,1.00) node{}
    +(8.33,1.00) node{}
    +(8.89,1.00) node[fill=black]{}
    +(8.89,0.50) node{}
    +(8.33,0.50) node{}
    +(7.78,0.50) node{}
    +(7.22,0.50) node{}
    +(6.67,0.50) node{}
    +(6.11,0.50) node{}
    +(5.56,0.50) node{}
    +(5.56,0.00) node{}
    +(6.11,0.00) node{}
    +(6.67,0.00) node[fill=black]{}
    +(7.22,0.00) node{}
    +(7.78,0.00) node{}
    +(8.33,0.00) node[fill=black]{}
    +(8.89,0.00) node{}
    +(5.00,2.25) node{}
    +(5.00,1.75) node[fill=black]{}
    +(5.56,1.75) node[fill=black]{}
    +(6.11,1.75) node{}
    +(6.67,1.75) node[fill=black]{}
    +(7.22,1.75) node[fill=black]{}
    +(7.78,1.75) node{}
    +(8.33,1.75) node[fill=black]{}
    +(8.89,1.75) node[fill=black]{}
    +(8.89,2.25) node{}
    +(8.33,2.25) node{}
    +(7.78,2.25) node[fill=black]{}
    +(7.22,2.25) node{}
    +(6.67,2.25) node{}
    +(6.11,2.25) node[fill=black]{}
    +(5.56,2.25) node{}
    +(0.56,2.25) node{}
    +(0.56,1.75) node{}
    +(1.11,1.75) node{}
    +(1.67,1.75) node[fill=black]{}
    +(2.22,1.75) node[fill=black]{}
    +(2.78,1.75) node{}
    +(3.33,1.75) node{}
    +(3.89,1.75) node{}
    +(3.89,2.25) node{}
    +(3.33,2.25) node[fill=black]{}
    +(2.78,2.25) node{}
    +(2.22,2.25) node{}
    +(1.67,2.25) node{}
    +(1.11,2.25) node{}
    +(1.7,3.2) node[rectangle, draw=white!0, fill=white!100]{$EPN(G_{8,3})=16$}   
    +(6.7,3.2) node[rectangle, draw=white!0, fill=white!100]{$IPN(G_{8,3})=13$}
    +(1.7,1.45) node[rectangle, draw=white!0, fill=white!100]{$EIPN(G_{8,3})=22$}   
    +(6.7,1.45) node[rectangle, draw=white!0, fill=white!100]{$ISPN(G_{8,3})=14$}
    +(1.7,-0.3) node[rectangle, draw=white!0, fill=white!100]{$EISPN(G_{8,3})=24$}   
    +(6.7,-0.3) node[rectangle, draw=white!0, fill=white!100]{$ESPN(G_{8,3})=22$}

  };
\end{tikzpicture}
\end{center}
\vskip -0.6 cm \caption{Maximizing private neighbors on $P_8 \cp P_3$} \label{fig:PN}
\end{figure}
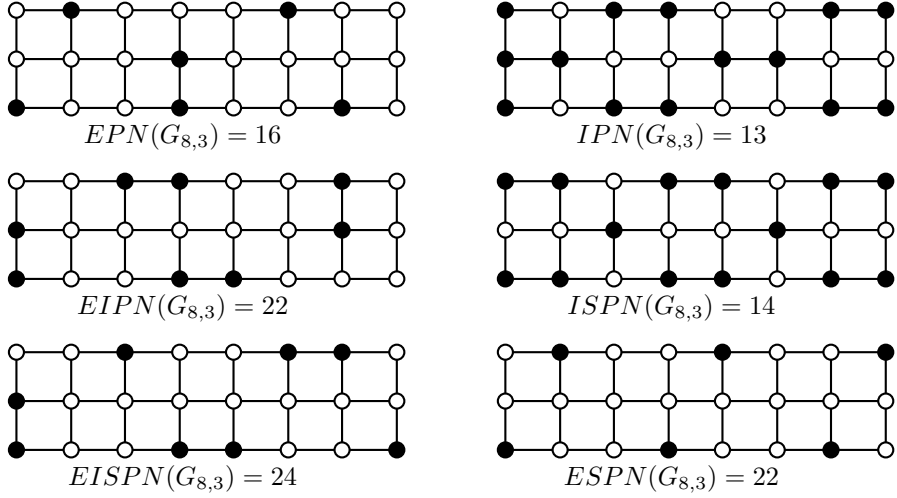

By their very definitions these seven private neighbor parameters are naturally related to one another.  
The following inequalities follow directly from the definitions.
\begin{prop}
If $G$ is any graph, then 
\begin{enumerate}
\item $SPN(G) \le ESPN(G) \le EISPN(G)$,
\item $SPN(G) \le ISPN(G) \le EISPN(G)$,
\item $IPN(G) \le ISPN(G) \le EISPN(G)$,
\item $IPN(G) \le EIPN(G) \le EISPN(G)$,
\item $EPN(G) \le ESPN(G) \le EISPN(G)$, and
\item $EPN(G) \le  EIPN(G) \le EISPN(G)$.
\end{enumerate}  
\end{prop}

The private neighbor maximization parameters we have defined on a graph $G$ measure the largest subset of vertices that are the 
corresponding type(s) of private neighbors.  The irredundance-type invariants are the maximum size of
a subset of vertices, each of which has the specified  type(s) of private neighbors. For example, if $D$ is an open-open irredundant 
set in $G$ with $|D|=OOIR(G)$ and $v\in D$, then  $v$ has an external private neighbor or an internal private neighbor with respect to $D$.
It follows that $OOIR(G)=|D|\leq EIPN(G)$, which is inequality $(5)$ in the following proposition.  Justification of the other six inequalities is similar.  

\begin{prop} \label{prop:inequalities}
If $G$ is any graph, then
\begin{enumerate}
\item $\alpha(G) = SPN(G)$,
\item $2\alpha^*(G) \le IPN(G)$, 
\item $OIR(G) \le EPN(G)$,
\item $IR(G) \le ESPN(G)$,
\item $OOIR(G) \le EIPN(G)$,
\item $\alpha_1(G) \le ISPN(G)$,
\item $COIR(G) \le EISPN(G)$.
\end{enumerate}
\end{prop}

We note that it is not the case that in every graph $G$ there exists a strong matching $M$ such that $2|M|=IPN(G)$.  
For example, the graph $G$ in Figure~\ref{fig:IPNex} has $IPN(G)=6$ and $2\alpha^*(G)=4$.  The white vertices form an $IPN(G)$-set, and it is easy to see that $\alpha^*(G)=2$.

\begin{figure}[htb]
\begin{center}
\begin{tikzpicture}[scale=.9,style=thick,x=1cm,y=1cm]
\def\vr{4pt} 
\path (-2,2) coordinate (u);  \path (-1,1) coordinate (x3); \path (-1,2) coordinate (x2); \path (-1,3) coordinate (x1);
\path (0,2) coordinate (a);  \path (1,2) coordinate (b); \path (2,1) coordinate (y3); \path (2,2) coordinate (y2);
\path (2,3) coordinate (y1);  \path (3,2) coordinate (v);
\foreach \i in {1,...,3} {\draw (a)--(x\i); \draw (b)--(y\i); }
\foreach \i in {1,3} {\draw (u)--(x\i); \draw (v)--(y\i); }
\draw (a) -- (b); 
\foreach \i in {1,...,3}
{  \draw (y\i)  [fill=white] circle (\vr); \draw (x\i)  [fill=white] circle (\vr);}
\draw (a)  [fill=white] circle (\vr); \draw (b)  [fill=white] circle (\vr);
\draw (u)  [fill=black] circle (\vr); \draw (v)  [fill=black] circle (\vr);
\end{tikzpicture}
\end{center}
\vskip -0.5 cm
\caption{$IPN(G)=6$ and $2\alpha^*(G)=4$} \label{fig:IPNex}
\end{figure}
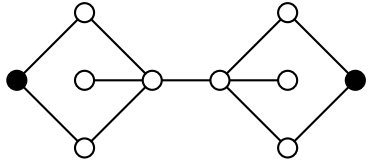

On the other hand, each of the other six maximum private neighbor invariants  can be realized in any graph using a set of vertices that has the corresponding irredundance-type property expressed in Proposition~\ref{prop:inequalities}.  We hasten to add, however, that the set 
having the corresponding irredundance-type property may not be a maximum such set.  For example, consider the relationship expressed in Proposition~\ref{prop:irrset} below.  It can be shown that $ESPN(K_{4,7})=9$ and $IR(K_{4,7})=7$, but the only irredundant sets $U$ that 
have $E(U)+S(U)=9$ have cardinality $2$.

\begin{prop} \label{prop:openirrset}
If $A$ is an $EPN(G)$-set, then there exists an $EPN(G)$-set $B$, such that $B$ is open irredundant and  $B\subseteq A$. 
\end{prop}
\begin{proof}
Suppose $A$ is not open irredundant and $E(A)=EPN(G)$.  Let $a \in A$ such that $a$ does not have an external private neighbor with
respect to $A$ and let $A'=A-\{a\}$.  It is clear that $E(A) \le E(A')$, and since $E(A)=EPN(G)$ we have $E(A)= E(A')$.
Therefore, by repeating this process, we conclude that there exists an $EPN(G)$-set $B$ that is open irredundant and a subset of $A$. \qed 
\end{proof}

\begin{prop} \label{prop:irrset}
If $A$ is an $ESPN(G)$-set, then there exists an $ESPN(G)$-set $B$, such that $B$ is irredundant and  $B\subseteq A$. 
\end{prop}
\begin{proof}
Suppose $A$ is not an irredundant set in $G$.  Let $a \in A$ be such that $N[a] - N[A-\{a\}]=\emptyset$ (or equivalently $N[a] \subseteq N[A-\{a\}]$). It follows that $a$ is not a self private neighbor with respect to $A$.   Let $A'=A-\{a\}$.    We claim that $A'$ is an $ESPN(G)$-set.  It is clear that $S(A')\ge S(A)$. Let $w$ be an external private neighbor with respect to $A$.  Then, there exists a unique $x \in A$ such that $\{x\}=N[w] \cap A$.  Since $N[a] \subseteq N[A-\{a\}]$, it follows that $x \neq a$ and hence $w$ is an external private neighbor with respect to $A'$.   Therefore, 
\[ESPN(G)=S(A)+E(A)\le S(A')+E(A')\le ESPN(G)\,.\]  
By repeating this process, we conclude that there exists a $ESPN(G)$-set that is irredundant and a subset of $A$.  \qed
\end{proof}

\begin{prop} \label{prop:open-openirrset}
If $A$ is a $EIPN(G)$-set, then there exists a $EIPN(G)$-set $B$, such that $B$ is open-open irredundant and  $B\subseteq A$. 
\end{prop}
\begin{proof}
Suppose that $A$ is an $EIPN(G)$-set that is not open-open irredundant.  That is, there exists $v \in A$ such that $N(v) \subseteq N(A-\{v\})$.  Thus, if $v'\in N(v) \cap A$, then $|N(v')\cap A| \ge 2$.  Similarly, if $v'' \in N(v) \cap (V(G)-A)$, then 
$|N(v'')\cap A| \ge 2$.  Let $A'=A-\{v\}$.  Suppose $x$ is an external private neighbor of $A$.  That is, $x \notin A$ and $|N(x) \cap A|=1$.  Since $v$ does not have an external private neighbor with respect to $A$, it follows that $|N(x) \cap A'|=1$.  Now, suppose that
$y$ is an internal private neighbor of $A$.  This means that there exists $z \in A$ such that $N(y) \cap A=\{z\}$, which implies that
$z \neq v$ since every neighbor of $v$ in $A$ has at least two neighbors in $A$.  It follows that $N(y) \cap A'=\{z\}$ and hence $y$
is an internal private neighbor with respect to $A'$.  We infer that 
\[EIPN(G)=E(A)+I(A) \le E(A')+I(A')\le EIPN(G)\,,\] and it follows that $A'$ is an $EIPN(G)$-set.  
Therefore, by repeating this process, we conclude that there exists an $EIPN(G)$-set $B$ that is open-open irredundant and a subset of $A$.
\qed
\end{proof}

\medskip

\begin{prop} \label{prop:1-dependent}
If $A$ is a $ISPN(G)$-set, then there exists a $ISPN(G)$-set $B$, such that $\alpha_1(G)=|B|$ and  $B\subseteq A$. 
\end{prop}
\begin{proof}
Let $A$ be an $ISPN(G)$-set such that $\delta(G[A]) \ge 2$.  Let $a \in A$ such that $a$ has at least two neighbors in $G[A]$, 
and let $A'=A-\{a\}$.  If $w$ is a self private neighbor with respect to $A$, then $w\neq a$ since $\deg_{G[A]}(a) \ge 2$.  It follows
that $w$ is a self private neighbor with respect to $A'$.  Now, suppose that $w$ is an internal private neighbor with respect to $A$.  That is, there exists $y \in A$ such that $N(w) \cap A=\{y\}$.  If $y=a$, then $w$ is a self private neighbor with respect to $A'$.  On
the other hand, if $y \neq a$, then $w$ is an internal private neighbor with respect to $A'$.  We infer that $A$ and $A'$ have the same
number of self and internal private neighbors.  Therefore, by repeating the process of removing vertices that have degree at least $2$
in the subgraph induced by the $ISPN(G)$-set we arrive at a $1$-dependent set that is also an $ISPN(G)$-set.  \qed
\end{proof}

\begin{prop} \label{prop:closed-openirrset}
If $A$ is a $EISPN(G)$-set, then there exists a $EISPN(G)$-set $B$, such that $B$ is closed-open irredundant and  $B\subseteq A$. 
\end{prop}
\begin{proof}
Suppose that $A$ is an $EISPN(G)$-set that is not closed-open irredundant.  This means that there exists $v \in A$ such that 
$N[v] \subseteq N(A-\{v\})$.  It follows that $v$ is not isolated in the subgraph $G[A]$ and $N(v) \subseteq N(A-\{v\})$.  Let $A'=A-\{v\}$.
As in the proof of Proposition~\ref{prop:open-openirrset} we see that $E(A) \leq E(A')$ and $I(A) \leq I(A')$.  Also, since $v$ is not isolated in $G[A]$, we have $S(A)\leq S(A')$.  Therefore,
\[ EISPN(G)=E(A)+I(A)+S(A) \le E(A')+I(A')+S(A')\le EISPN(G)\,.\]  It follows that $A'$ is an $EISPN(G)$-set.
Therefore, by repeating this process, we conclude that there exists an $EISPN(G)$-set $B$ that is closed-open irredundant and a subset of $A$.  \qed
\end{proof}

\section{Maximizing private neighbors in graph classes}  \label{sec:classes}

In this section we consider some of the well-known classes of graphs and determine the maximum
number of private neighbors they possess.  The proofs of the following closed form formulas for paths, cycles and complete bipartite 
graphs are straightforward, although the proofs are often case studies.  Since $SPN(G)=\alpha(G)$, the values for maximizing self private neighbors are omitted.

\begin{prop} \label{prop:paths}
If $n\ge 2$ is a positive integer, then
\begin{enumerate}
\item $IPN(P_n)=2\lceil\frac{n-1}{3}\rceil$.  
\item $EPN(P_n)= \begin{cases} 2\lfloor\frac{n}{3}\rfloor, &\text{if } n\not\equiv 2 \pmod{3};\\ 
                                      2\lfloor\frac{n}{3}\rfloor+1, &\text{if } n\equiv 2 \pmod{3}. 
                 \end{cases}$
\item $EIPN(P_n)= \begin{cases} 4\lfloor\frac{n}{4}\rfloor, &\text{if } n \pmod{4}\in \{0,1\};\\ 
                                n, &\text{if } n \pmod{4}\in \{2,3\}.
                 \end{cases}$
\item $ESPN(P_n)=EISPN(P_n)=n$.
\item $ISPN(P_n)= \begin{cases} 2\lceil\frac{n-1}{3}\rceil+1, &\text{if } n\equiv 1 \pmod{3};\\ 
                                      2\lceil\frac{n-1}{3}\rceil, &\text{if } n\not\equiv 1 \pmod{3}. 
                 \end{cases}$
\end{enumerate}
\end{prop}

\begin{prop}  \label{prop:cycles}
If $n \ge 3$ is a positive integer, then
\begin{enumerate}
\item $IPN(C_n)=EPN(C_n)=2 \lfloor\frac{n}{3} \rfloor$.
\item $EIPN(C_n)= \begin{cases} 4\lceil\frac{n-2}{4}\rceil, &\text{if } n\not\equiv 3 \pmod{4};\\ 
                                n-1, &\text{if } n \equiv 3 \pmod{4}.
                 \end{cases}$
\item $ESPN(C_n)= \begin{cases} n, & \text{if } n\equiv 0 \pmod{3};\\ 
                                n-1, &\text{if } n\not \equiv 0 \pmod{3}.
                 \end{cases}$ 
\item $ISPN(C_n)=  \begin{cases} 2\lceil\frac{n}{3}\rceil +1, & \text{if } n\equiv 2 \pmod{3};\\ 
                                2 \lfloor\frac{n}{3}\rfloor, &\text{if } n\not \equiv 2 \pmod{3}.
                 \end{cases}$ 
\item $EISPN(C_n)=   \begin{cases} 4, & \text{if } n=5;\\ 
                                n, &\text{if } n\neq 5.
                 \end{cases}$            
\end{enumerate}
\end{prop}

\begin{prop} \label{prop:completebipartite}
If $p$ and $q$ are positive integers such that $p \le q$, then  
\begin{itemize}
\item $IPN(K_{p,q})=ISPN(K_{p,q})=q$.
\item $EPN(K_{p,q})=   \begin{cases} q, & \text{if } p=1;\\ 
                                p+q-2, &\text{if } p\ge 2.
                 \end{cases}$ 
\item $EIPN(K_{p,q})=EISPN(K_{p,q})=p+q$.
\item $ESPN(K_{p,q})=   \begin{cases} q+1, & \text{if } p=1;\\ 
                                \max\{1+p, p+q-2\}, &\text{if } p\ge 2.
                       \end{cases}$ 
\end{itemize}
\end{prop}

The class of grid graphs are well-studied for some of the irredundance-type invariants.  In the following tables we 
include computed values for the maximum number of private neighbors of some small grids.   The resulting numbers suggest several conjectures or open problems, some of which we list in Section~\ref{sec:open}.  We then prove the exact values
of the maximum number of external or self private neighbors in $n \times 2$ grids.

\begin{table}[ht] 
\begin{center}
\begin{tabular}{|c||c|c|c|c|c|c|c|c|}
  \hline
   & 2 & 3 & 4 & 5 & 6 & 7 & 8 & 9 \\ \hline \hline
  2 & 2 & 4 & 4 & 6 & 6 & 8 & 8 & 10 \\ \hline
  3 &  & 4 & 7 & 8 & 10 & 12 & 13 & 15 \\  \hline
  4 &  &  & 8 & 10 & 12 & 14 & 16 & 18 \\  \hline
  \hline
\end{tabular}
\caption{$IPN(P_n \cp P_m)$} \label{table:IPN}
\end{center}
\end{table}
\newpage  

\begin{table}[h]
\begin{center}
\begin{tabular}{|c||c|c|c|c|c|c|c|c|}
  \hline
   & 2 & 3 & 4 & 5 & 6 & 7 & 8 & 9 \\ \hline \hline 
  2 & 2 & 4 & 5 & 7 & 8 & 10 & 11 & 13 \\  \hline 
  3 &  & 6 & 8 & 10 & 12 & 15 & 16 & 19 \\  \hline 
  4 &  &  & 12 & 14 & 17 & 20 &  &  \\  \hline 
  \hline
\end{tabular}
\caption{$EPN(P_n \cp P_m)$} \label{table:EPN}
\end{center}
\end{table}

\begin{table}[h]
\begin{center}
\begin{tabular}{|c||c|c|c|c|c|c|c|c|}
  \hline \hline
   & 2 & 3 & 4 & 5 & 6 & 7 & 8 & 9 \\ \hline \hline
  2 & 4 & 6 & 8 & 10 & 12 & 14 & 16 & 18 \\ \hline
  3 &  & 8 & 10 & 14 & 16 & 20 & 22 & 25 \\ \hline
  4 &  &  & 16 & 18 & 24 & 28 &  &  \\ \hline
  \hline
\end{tabular}
\caption{$EIPN(P_n \cp P_m)$} \label{table:EIPN}
\end{center}
\end{table}

\begin{table}[h]
\begin{center}
\begin{tabular}{|c||c|c|c|c|c|c|c|c|}
  \hline
   & 2 & 3 & 4 & 5 & 6 & 7 & 8 & 9  \\ \hline \hline
  2 & 3 & 6 & 7 & 10 & 11 & 14 & 15 & 18 \\ \hline
  3 &  & 8 & 11 & 14 & 16 & 19 & 22 & 25 \\ \hline
  4 &  &  & 16 & 18 & 23 & 27 &  &  \\ \hline
  \hline
\end{tabular}
\caption{$ESPN(P_n \cp P_m)$} \label{table:ESPN}
\end{center}
\end{table}

\begin{table}[h]
\begin{center}
\begin{tabular}{|c||c|c|c|c|c|c|c|c|}
  \hline
   & 2 & 3 & 4 & 5 & 6 & 7 & 8 & 9  \\  \hline \hline
  2 & 2 & 4 & 4 & 6 & 6 & 8 & 8 & 10 \\ \hline
  3 &  & 5 & 7 & 9 & 10 & 12 & 14 & 15 \\ \hline
  4 &  &  & 8 & 11 & 12 & 15 &  &  \\ \hline
  \hline
\end{tabular}
\caption{$ISPN(P_n \cp P_m)$} \label{table:ISPN}
\end{center}
\end{table}

\begin{table}[h]
\begin{center}
\begin{tabular}{|c||c|c|c|c|c|c|c|c|}
  \hline
   & 2 & 3 & 4 & 5 & 6 & 7 & 8 & 9  \\ \hline \hline
  2 & 4 & 6 & 8 & 10 & 12 & 14 & 16 & 18 \\ \hline
  3 &  & 8 & 12 & 14 & 18 & 21 & 24 & 27 \\ \hline
  4 &  &  & 16 & 20 & 24 & 28 &  &  \\ \hline
  \hline
\end{tabular}
\caption{$EISPN(P_n \cp P_m)$} \label{table:EISPN}
\end{center}
\end{table}

\begin{prop} \label{prop:prisms}
If $n$ is a positive integer such that $n \ge 2$, then 
\[ESPN(P_n \cp P_2)=\begin{cases} 2n, &\text{if $n$ is odd;}\\ 2n-1, &\text{if $n$ is even.} \end{cases}\,.\]
\end{prop}
\begin{proof}
Recall that $G_{n,2}=P_n \cp P_2$.
Let $V(P_n)=[n]$ with $E(P_n)=\{ij \,:\, i \in [n-1] \text{ and } j=i+1\}$, let $V(P_2)=[2]$, and let
\[U=\{(k,1)\,:\, k \equiv 1 \pmod{4} \text{ and } k \leq n\} \cup \{(k,2)\,:\, k \equiv 3 \pmod{4} \text{ and } k \leq n\}\,.\] 
If $n$ is odd, a straightforward check shows that the set $U$ is an efficient dominating set of $G_{n,2}$. Therefore, 
$ESPN(P_n \cp P_2)=2n$.  Now, assume 
that $n$ is even, say $n=2r$.  It is known that $\gamma(G_{2r,2})=r+1$; see Theorem 17.1 in \cite{hhh-2023}.  One can 
easily check that $S(U)=r$ and $E(U)=3r-1$, which implies that $ESPN(G_{n,2})\ge 2n-1$. 
We claim that $P_n \cp P_2$ does not have an efficient dominating set.  Suppose to the
contrary that $D=\{v_1,\ldots, v_k\}$ is such a dominating set.  It follows that $4r=p_1+p_2 +\cdots +p_k$,
where $p_i=|N[v_i]|\in \{3,4\}$ for each $i \in [k]$ . Since $G_{n,2}$ has only four vertices of degree 2, analysis shows that
$4r=p_1+p_2 +\cdots +p_k$ is possible only if $|\{i\in [k] \,:\, p_i=3 \}|=4$.  This implies 
that $\{(1,1),(1,2),(n,1),(n,2)\} \subseteq D$, which is a contradiction.  \qed
\end{proof}

\section{Generalized efficient graphs}

Suppose a graph $G$ has an efficient dominating set $D$.  In other contexts an efficient dominating set is called a \emph{perfect code}.  The graph $G$ is then said to be \emph{efficient}.  By definition, each vertex in $D$ is a self private neighbor with respect to $D$ (that is, $S(D)=|D|$), and every vertex in $V-D$ is an external private neighbor with respect to $D$ (that is, $E(D)=n(G)-|D|$).  Therefore, $ESPN(G)=S(D)+E(D)=n(G)$.   Let us say that such a graph $G$ is \emph{$ES$-efficient}.  The converse is also true.    
\begin{prop}
A graph $G$ is efficient if and only if $ESPN(G)=n(G)$.
\end{prop}

Suppose $G$ is an efficient graph and let $D = \{v_1, v_2, \ldots, v_k\} \subset V(G)$  be an efficient dominating set in $G$.  It follows that $\pi = \{N[v_1], N[v_2], \ldots, N[v_k]\}$ is a partition of $V(G)$.   Not every graph $G$ is $ES$-efficient. Two simple examples are the cycles $C_{4}$ and $C_{5}$, although it is easy to see that all cycles and paths having order congruent to zero modulo $3$
are $ES$-efficient.  

In~\cite{Jason} Hedetniemi, et al. defined two other types of efficient dominating sets. Using our current terminology they said a dominating set $D$ of $G$ is a \emph{$1,2$-efficient} dominating set if every vertex in $V-D$ is an external private neighbor with respect to $D$ (that is, $E(D)=n(G)-|D|$) and every vertex in $D$ is a self private or an internal private neighbor (that is, $|D|=S(D)+I(D)$) with respect to $D$. That is, $EISPN(G)=E(D)+S(D)+I(D)=n(G)$.  We will say that such a graph $G$ is \emph{$EIS$-efficient}. Klostermeyer and Goldwasser~\cite{kg-2006} showed that $G_{n,2}$ is $EIS$-efficient for every $n$.  All paths and all cycles except $C_5$ were shown to be $EIS$-efficient in~\cite{Jason}. Furthermore, for $3 \le m \le 5$, all grid graphs $G_{n,m}$ (with the exception of $G_{3,3}, G_{5,3}, G_{3,5}$ and $G_{9,5}$) were  shown to be $EIS$-efficient in~\cite{Jason}.  They also proved that each $G_{r,s}$ for $r+s$ odd and each $G_{n,6}$ for $n\pmod{7} \in \{0,1,2,4,6\}$ are $EIS$-efficient and left the determination for the two cases $n\pmod{7} \in \{3,5\}$ as open problems.

In addition, they defined a dominating set $D$ to be a \emph{total efficient} dominating set if  every vertex in $V-D$ is an external private neighbor with respect to $D$ (that is, $E(D)=n(G)-|D|$) and every vertex in $D$ has an internal private neighbor with respect to $D$ (that is, $|D|=I(D)$). Therefore, $EIPN(G)=E(D)+I(D)=n(G)$. We will say that such a graph is \emph{$EI$-efficient}. It is easy to see that a path is $EI$-efficient if and only if its order is not congruent to $1$ modulo $4$.  Klostermeyer and Goldwasser characterized the $EI$-efficient grid graphs.  See~\cite{kg-2006}. 

Every connected $EI$-efficient graph $G$ can be constructed as follows.  Let $\cF_1$ be the family of all double stars and let $\cF_2$ the family of all stars with at least two leaves.  For any positive integer $k$, select $H_1,\ldots,H_k$ from $\cF_1 \cup \cF_2$.  For each selection from $\cF_2$, consider the center and one leaf as being marked, and for each double star chosen from $\cF_1$ consider the two centers as being marked. The vertex set of $G$ is  $V(H_1) \cup \cdots \cup V(H_k)$.  The edges of $G$ are the edges of the selected graphs 
together with any collection of edges joining the unmarked vertices (that is, the leaves) so that the resulting graph is connected. 
Note that if $u_i$ and $v_i$ are the marked vertices of $H_i$, then $\{N[\{u_1,v_1\}], \ldots, N[\{u_k,v_k\}]\}$ is a partition of $V(G)$.
Although the above is a constructive characterization of the class of $EI$-efficient graphs, determining whether a given connected graph
is $EI$-efficient is most likely a difficult problem.

In an analogous way, the class of connected $EIS$-efficient graphs can be constucted.  Let $\cF_3$ be the family of all stars of order at least $2$.  For any positive integer $k$, select $H_1,\ldots,H_k$ from $\cF_1 \cup \cF_2 \cup \cF_3$.  For each selection from $\cF_2$, consider the center and one leaf as being marked; for each double star chosen from $\cF_1$ consider the two centers as being marked; for each
star chosen from $\cF_3$ mark the center.  Proceed as in the case of $EI$-efficient graphs and add any set of edges joining leaves in
the disjoint union $V(H_1) \cup \cdots \cup V(H_k)$ so long as the resulting graph is connected.  For example, $K_{r,s}$, where $2 \le r \le s$, is $EIS$-efficient and can be constructed as shown by starting with the single double star $S(r-1,s-1)$.

In a similar vein there are three additional kinds of efficient graphs, although their structure is rather trivial.   A graph $G$ is \emph{$I$-efficient} (respectively, \emph{$S$-efficient}, \emph{$IS$-efficient}) if there is a set $U \subseteq V(G)$ such that every vertex of $G$ is an internal (respectively, a self, an internal or self) private  neighbor  with respect to $U$.  It is clear that the only $I$-efficient graphs are disjoint unions of complete subgraphs of order $2$; the only $S$-efficient graphs are sets of isolated vertices; the only $IS$-efficient graphs are disjoint unions of isolated vertices and complete subgraphs of order $2$. 

\section{A lower bound for $IPN(T)$ for trees $T$}

In this section we establish a sharp lower bound for the maximum number of internal private neighbors in any nontrivial tree.  
Let $T$ be a tree of order at least $2$. Recall that $I(U)$ is the number of internal private neighbors in $T[U]$, the subgraph of $T$ induced by $U \subseteq V(T)$. That is, $I(U)$ is the number of vertices of degree $1$ in $T[U]$.

\begin{thm} \label{thm:lowerIPN}
For any tree $T$ of order $n \ge 2$, $IPN(T) \ge \frac{n}{2}$, and this bound is sharp.  
\end{thm}

\begin{proof} We proceed by induction on the order $n$ of the tree.
By inspection, it is easy to verify the truth of the statement for any tree of order $n \le 5$.
Assume that  $n \ge 5$ and that every tree $T'$ of order at most $n$ satisfies $IPN(T') \ge \frac{n(T')}{2}$.
Let $T$ be a tree of order $n+1$.  Consider a path in $T$ of length $diam(T)$.
Suppose first that $diam(T) = 2$.  It follows that $T = K_{1,n}$, and $IPN(K_{1,n})=n \ge (n+1)/2$.
Suppose next that $diam(T) = 3$.  In this case $T$ is a \emph{double star}, $T = S(r,s)$ that consists of two adjacent vertices $x$ and $y$, where $x$ is adjacent to $r$ leaves, $y$ is adjacent to $s$ leaves, and both $r$ and $s$ are positive integers.  It is easy to see that $IPN(T)=n-1 \ge (n+1)/2$.  

Finally, suppose that $diam(T)=p \ge 4$ and let $v_1,v_2, \ldots, v_p,v_{p+1}$ be a diametral path in $T$. Let $T'$ denote the subtree of $T - v_2v_3$ containing vertices $v_1$ and $v_2$ and let $T''$ be the subtree of  $T-v_2v_3$ that contains $v_3$.  Note that if $v_2$ has degree at least $2$ in $T'$, then $T'$ is a star.   By the induction hypothesis, we know that 
$IPN(T'')  \ge \frac{n(T'')}{2}$.  Let $U''$ be an $IPN(T'')$-set. Suppose first that $v_3 \notin U''$.  If $T'$ consists only of vertices $v_1$ and $v_2$ then $U = U'' \cup \{v_1,v_2\}$ shows that $IPN(T) \ge I(U) \ge (n+1)/2$.  On the other hand, if  $T' = K_{1,r}$ for some $r \ge 2$, then $U = U'' \cup V(T')$ shows that $IPN(T) \ge (n+1)/2$.  Suppose next that $v_3 \in U''$, $v_3$ has
an internal private neighbor in $T''[U'']$ but that $v_3$ is not itself an internal private neighbor in $T''[U'']$.
If $T'$ consists only of vertices $v_1$ and $v_2$ then $U = U'' \cup \{v_2\}$ shows that $IPN(T) \ge (n+1)/2$.  If $T' = K_{1,r}$ for some $r \ge 2$, then $U = U'' \cup V(T')$ shows that $IPN(T) \ge (n+1)/2$. Finally, suppose that $v_3$ is an internal private neighbor
in $T''[U'']$.  In this case, $U = (U'' - \{v_3\}) \cup V(T')$ shows that $IPN(T) \ge (n+1)/2$. 

To prove sharpness let $T_m=P_m \circ K_1$, the corona of $P_m$.  We claim that if $m$ is even, then $IPN(T_m)=m=\frac{1}{2}n(T_m)$.
For notational convenience let $V(T_m)=\cup_{i=1}^m\{u_i,v_i\}$ and $E(T_m)=\{u_iu_{i+1}\,:\, i\in [m-1]\}\cup \{u_iv_i\,:\, i \in [m]\}$.
Let $m$ be even, say $m=2k$, let $U$ be an $IPN(T_m)$-set and let $H$ be the subgraph of $T_m$ induced by $U$.  By definition, $IPN(T_m)=|\{x\in V(H)\,:\, x \text{ has degree $1$ in } H\}|$.  For each $j \in [k]$, let $B_j = U \cap \{u_{2j-1},v_{2j-1},u_{2j},v_{2j}\}$.  Since $B_j$ is a subgraph of a path of order $4$, it is clear that $B_j$ has at most $2$ vertices of degree $1$ in $H$.  Therefore, $IPN(T_m)\le 2k$.
By applying the first part of the proof of the theorem it follows that $IPN(T_m)=\frac{n(T_m)}{2}$.
\qed
\end{proof}

\section{Perfect dominating sets in graphs}

In 1993, Cockayne, Hartnell, Hedetniemi and Laskar \cite{Cockayne93} defined a set $S \subset V$ to be a {\it perfect dominating set} if for every vertex $u \in V - S$, $|N(u) \cap S| = 1$.  Notice in this definition that no conditions are placed on the vertices in $S$; the subgraph $G[S]$ induced by $S$ can be any graph.  They defined the \emph{perfect domination number} $\gamma_p(G)$ to equal the minimum cardinality of a perfect dominating set, and the \emph{upper perfect domination number} $\Gamma_p(G)$ to equal the maximum cardinality of a perfect dominating set.  They pointed out that for some graphs the only perfect dominating set is vacuously the entire set $V(G)$.

Also in 1993, Bernhard, Hedetniemi and Jacobs \cite{Bernhard93} defined, for any set $U \subset V$, $ED(U) = \{ v : v \in V - U$  and $|N(v) \cap U| = 1\}$, that is, $ED(U)$ equals the number of vertices in $V - U$ which are \emph{efficiently} dominated by vertices in $U$.  This is equivalent to saying that $ED(U)$ equals the number of vertices in $V - U$ which are external private neighbors with respect to the set $U$.  They went on to define the \emph{efficient domination number} to equal the minimum cardinality of a dominating set $U$ for which $ED(U) = V - U$.  Thus, the  efficient domination number in \cite{Bernhard93} is the same thing as the perfect domination number in \cite{Cockayne93}. We should point out that the \cite{Bernhard93} definition of an efficient dominating set is not the same as the definition we have given of an efficient dominating set, the difference being that in the \cite{Bernhard93} definition no conditions are placed on the vertices in $U$.

The following inequality is clear from the definitions.

\begin{prop}
For any graph $G$ of order $n$, $n - \gamma_p(G) \le EPN(G)$.
\end{prop}

As defined by Dunbar et al. \cite{Dunbar95},  a set $U \subset V$ is called a {\it total perfect dominating set} if for every vertex $v \in V(G)$, $|N(v) \cap U| = 1$.  Notice that in a total perfect dominating set $U$, the subgraph $G[U]$ induced by the vertices in $U$ is, again, a disjoint union of $K_2$ subgraphs.  Define the {\it total perfect domination number} $\gamma_{tp}(G)$ and the {\it upper total perfect domination number} $\Gamma_{tp}(G)$ to equal the minimum and maximum cardinality of a total perfect dominating set of $G$, if such a set exists.  Notice that if a total perfect dominating set $U$ exists, then every vertex $v \in V$ is either an external private neighbor with respect to $U$ or an internal private neighbor with respect to $U$ and the following must be true.

\begin{prop}
If a graph $G$ of order $n$ has a total perfect dominating set $U$, then $E(U)+I(U) = EIPN(G) = n$.
\end{prop}

\section{Upper private domination}

In 1979 Bollob\'{a}s and Cockayne \cite{Bollobas79} showed that every graph $G$ without isolated vertices has at least one minimum dominating set $U$ such that every vertex in $U$ has at least one external private neighbor with respect to $U$.  Based on this result, in 1990 Hedetniemi, Hedetniemi and Jacobs \cite{Hedetniemi90} defined a dominating set $U$ to be a {\it private dominating set} if every vertex $v \in U$ has at least one external private neighbor.   From our previous definitions, such a set could also be called an {\it open irredundant dominating set}. The minimum cardinality of a private dominating set in a graph $G$ is the {\it private domination number}, denoted by $\gamma_{pvt}(G)$, while the {\it upper private domination number} $\Gamma_{pvt}(G)$ equals the maximum cardinality of a minimal private dominating set in a graph.  The next results follow from \cite{Bollobas79} and the definitions.

\begin{prop}
For any graph without isolated vertices, 
$$\gamma(G) = \gamma_{pvt}(G) \le \Gamma_{pvt}(G) \leq EPN(G).$$
\end{prop}

\section{Open problems}  \label{sec:open}
We note that $ESPN(G)\ge IR(G)=\alpha(G)\ge \frac{1}{2}n(G)$ for any bipartite graph $G$.  In particular, $ESPN(T) \ge \frac{1}{2}n(T)$
for every tree.  We suspect the lower bound in terms of the order of a tree is larger.  The following class of trees shows that a 
sharp lower bound could be $\frac{4}{5}n(T)$. 
For each positive integer $k \ge 2$, let $T_k$ be the tree constructed from $P_k$ by attaching two leaves adjacent to each
vertex of $P_k$. In addition, to each vertex of the original path identify a leaf from a path of order $3$. It can be shown that 
$ESPN(T_k)=\frac{4}{5}n(T_k)$.  For example, the black vertices in Figure~\ref{fig:ESPN} illustrate an $ESPN(T_4)$-set and 
$ESPN(T_4)=16$.

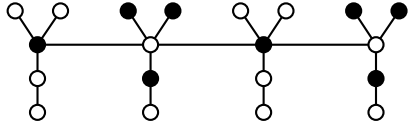
\begin{figure}[htb]
\tikzstyle{every node}=[circle, draw, fill=black!0, inner sep=0pt,minimum width=.2cm]
\begin{center}
\begin{tikzpicture}[thick,scale=.9]
  \draw(0,0) { 
    +(0.00,1.50) -- +(0.33,1.00)
    +(0.67,1.50) -- +(0.33,1.00)
    +(0.33,1.00) -- +(0.33,0.00)
    +(1.67,1.50) -- +(2.00,1.00)
    +(2.00,1.00) -- +(2.33,1.50)
    +(2.00,1.00) -- +(2.00,0.00)
    +(3.33,1.50) -- +(3.67,1.00)
    +(3.67,1.00) -- +(4.00,1.50)
    +(3.67,1.00) -- +(3.67,0.00)
    +(5.00,1.50) -- +(5.33,1.00)
    +(5.33,1.00) -- +(5.67,1.50)
    +(5.33,1.00) -- +(5.33,0.50)
    +(5.33,0.50) -- +(5.33,0.00)
    +(0.33,1.00) -- +(2.00,1.00)
    +(2.00,1.00) -- +(3.67,1.00)
    +(3.67,1.00) -- +(5.33,1.00)
    +(0.00,1.50) node{}
    +(0.67,1.50) node{}
    +(1.67,1.50) node[fill=black]{}
    +(2.33,1.50) node[fill=black]{}
    +(3.33,1.50) node{}
    +(4.00,1.50) node{}
    +(5.00,1.50) node[fill=black]{}
    +(5.67,1.50) node[fill=black]{}
    +(0.33,1.00) node[fill=black]{}
    +(0.33,0.50) node{}
    +(0.33,0.00) node{}
    +(2.00,1.00) node{}
    +(2.00,0.50) node[fill=black]{}
    +(2.00,0.00) node{}
    +(3.67,1.00) node[fill=black]{}
    +(3.67,0.50) node{}
    +(3.67,0.00) node{}
    +(5.33,1.00) node{}
    +(5.33,0.50) node[fill=black]{}
    +(5.33,0.00) node{}
    
  };
\end{tikzpicture}
\end{center}
\vskip -0.6 cm \caption{The tree $T_4$.} \label{fig:ESPN}
\end{figure}

\begin{prob}
Find the largest positive constant $c$ such that $ESPN(T) \ge c\cdot n(T)$ for every tree $T$. 
\end{prob}

Based on the data presented in Section~\ref{sec:classes} we pose the following conjectures.
\begin{conj} \label{conj:EIPNgrids}
If $m \ge 2$, then $EIPN(P_2 \cp P_m)=2m$.
\end{conj}
The truth of Conjecture~\ref{conj:EIPNgrids} would imply that $EISPN(P_2 \cp P_m)=2m$.  In addition, we pose the following:
\begin{conj}
If $k \ge 1$, then $EISPN(P_3 \cp P_{2k})=6k$. If $m \ge 2$, then $EISPN(P_4 \cp P_m)=4m$.
\end{conj}

\begin{prob}
Construct algorithms to compute each of  $IPN(T)$, $EPN(T)$, $ESPN(T)$, $EIPN(T)$, $ISPN(T)$, and $EISPN(T)$ for a tree $T$.
\end{prob}

\begin{prob}
Determine the complexity  of computing each of  $IPN(G)$, $EPN(G)$, $ESPN(G)$, $EIPN(G)$, $ISPN(G)$, and $EISPN(G)$ for a graph $G$.
\end{prob}


\begin{thebibliography}{99}

\bibitem{Bernhard93}
P.J. Bernhard, S.T. Hedetniemi, D.P. Jacobs, Efficient sets in graphs.  {\it Discrete Appl. Math.} {\bf 44} (1993), 99--108.



\bibitem{Bollobas79}
B. Bollob\'{a}s, E.J. Cockayne, Graph-theoretic parameters concerning domination, independence, and irredundance. {\it J. Graph Theory} {\bf 3} (1979), 241--249.

\bibitem{Cameron-1989}  K. Cameron, Induced matchings.
First Montreal Conference on Combinatorics and Computer Science, 1987.
{\it Discrete Appl. Math.}  {\bf 24}  (1989),  no. 1-3, 97--102.

\bibitem{Cockayne99}
E.J. Cockayne, Generalized irredundance in graphs: hereditary properties and Ramsey numbers. {\it J. Combin. Math. Combin. Comput.} {\bf 31} (1999), 15--31.

\bibitem{Cockayne93}
E.J. Cockayne, B.L. Hartnell, S.T. Hedetniemi, R.C. Laskar, Perfect domination in graphs.  {\it J. Combin. Inform. System Sci.} {\bf 18} (1993), 136--148.

\bibitem{Cockayne78}
E.J. Cockayne, S.T. Hedetniemi, D.J. Miller, Properties of hereditary hypergraphs and middle graphs. {\it Canad. Math. Bull.} {\bf 21} (1978), 461--468.

\bibitem{Dunbar95}
J.E. Dunbar, F.C. Harris, S.M. Hedetniemi, S.T. Hedetniemi,  A.A. McRae, R.C. Laskar, Nearly perfect sets in graphs.  {\it Discrete Math.} {\bf 138} (1995), 229--246.

\bibitem{Farley83}
A.M. Farley, N. Shachum, Senders in broadcast networks: open irredundancy in graphs. {\it Congr. Numer.} {\bf 38} (1983), 47--57.

\bibitem{Fellows94}
M. Fellows, G. Fricke, S. Hedetniemi, D. Jacobs, The private neighbor cube.  {\it SIAM J. Discrete Math.} {\bf 7} (1994), no. 1, 41--47.

\bibitem{Finbow03}
S. Finbow, {\it Generalisations of irredundance in graphs}. PhD Thesis, University of Victoria, Canada (2003), 158 pp. http:\\hdl.handle.net/1828/7913

\bibitem{fj-1985} J.F. Fink,  M.S. Jacobson, n-domination in graphs.
 {\it Graph theory with applications to algorithms and computer science (Kalamazoo, Mich., 1984)}, 
 283--300, Wiley-Intersci. Publ., Wiley, New York,  1985. 

\bibitem{hhh-2023}  T.W. Haynes, S.T. Hedetniemi, M.A. Henning,  {\it Domination in graphs: core concepts}.
Springer Monographs in Mathematics. Springer, Cham,  [2023],  xx+644 pp. 

\bibitem{Jason}
J.T. Hedetniemi, K.D. Hedetniemi, S.M. Hedetniemi, S.T. Hedetniemi, 1,2-efficiency in grid graphs, {\it Utilitas Math.} {\bf 124} (2025), 171--187.

\bibitem{Hedetniemi90}
S.M. Hedetniemi, S.T. Hedetniemi, D.P. Jacobs, Private domination: theory and algorithms.  {\it Congr. Numer.} {\bf 79} (1990), 147--157.


\bibitem{Hedetniemi21}
S.T. Hedetniemi, A.A. McRae, R. Mohan, The private neighbor concept. {\it Dev. Math.} {\bf 66},
Springer, Cham, 2021, 183--218.

\bibitem{kg-2006} W.F. Klostermeyer, J.L. Goldwasser, Total perfect codes in grid graphs. {\it Bull. Inst.
Comb. Appl.} {\bf 46} (2006), 61--68.


\bibitem{Mynhardt21}
C.M. Mynhardt, A. Roux, Irredundance.  {\it Dev. Math.} {\bf 66},
Springer, Cham, 2021, 135--181.

\end{thebibliography}
\end{document}